\documentclass[12pt]{amsart}

\usepackage{amsfonts}
\usepackage{amsmath}
\usepackage{amssymb}

\usepackage{enumerate}

\usepackage{tikz}

\usepackage[a4paper,margin=2cm]{geometry}

\newtheorem{theorem}{Theorem}
\newtheorem{proposition}[theorem]{Proposition}
\newtheorem{cor}[theorem]{Corollary}
\newtheorem{lemma}[theorem]{Lemma}

\newtheorem{Ex}{Example}

\DeclareMathOperator{\STAB}{STAB} 
\DeclareMathOperator{\conv}{conv} 

\newcommand{\R}{\mathbb{R}} 
\newcommand{\compG}[1]{G(#1)} 
\newcommand{\incompG}[1]{\overline{G}(#1)} 
\newcommand{\ent}[1]{H(#1)} 
\newcommand{\cent}[1]{H(\overline{#1})} 
\newcommand{\comp}[1]{\overline{#1}} 

\newcommand{\epoch}{E}
\newcommand{\epochA}{\Psi}
\newcommand{\epochB}{\Omega}
\newcommand{\epochAsize}{\psi}
\newcommand{\epochBsize}{\omega}

\title{Poset Entropy versus Number of Linear Extensions: the Width-$2$ Case}
\date{}
\author{Samuel Fiorini \and Selim Rexhep}

\begin{document}

\maketitle

\begin{abstract}
Kahn and Kim (J.\ Comput.\ Sci., 1995) have shown that for a finite poset $P$, the entropy of the incomparability graph of $P$ (normalized by multiplying by the order of $P$) and the base-$2$ logarithm of the number of linear extensions of $P$ are within constant factors from each other. The tight constant for the upper bound was recently shown to be $2$ by Cardinal, Fiorini, Joret, Jungers and Munro (Combinatorica, 2013). Here, we refine this last result in case $P$ has width $2$: we show that the constant can be replaced by $2-\varepsilon$ if one also takes into account the number of connected components of size~$2$ in the incomparability graph of $P$. Our result leads to a better upper bound for the number of comparisons in algorithms for the problem of sorting under partial information.
\end{abstract}

\section{Introduction} \label{sec:intro}
 
The entropy of a graph is an information theoretic concept introduced by K\"orner in 1973 \cite{Ko}. Since then, links with many interesting combinatorial objects have been found, see the survey paper of Simonyi \cite{Sim} for more information. 

In this paper, we consider the case in which the graph is the incomparability graph $\incompG{P}$ of a (finite) poset $P$. We denote by $\cent{P} := H(\incompG{P})$ the entropy of this graph. Kahn and Kim~\cite{KK} have proved that $|P| \cdot \cent{P}$ is within a constant of $\log e(P)$, the base-$2$ logarithm of the number of linear extensions of $P$. (Throughout this paper, $\log$ denotes the base-$2$ logarithm).

\begin{theorem}[Kahn and Kim~\cite{KK}] \label{thm:factor_11} For every poset $P$:
$$
\log e(P) \leqslant |P| \cdot \cent{P} \leqslant c_0 \log e(P)
$$
for $c_0 = (1 + 7 \log \mathrm{e}) \simeq 11.1$.
\end{theorem}

Cardinal, Fiorini, Joret, Jungers and Munro \cite{SUPIjournal} improved the constant in the upper bound to~$2$. This is tight since if $P$ is a two-elements antichain we have $|P| \cdot \cent{P} = 2$ and $\log e(P) = 1$.

\begin{theorem}[Cardinal et al.~\cite{SUPIjournal}] \label{thm:factor_2} For every poset $P$:
$$
|P| \cdot \cent{P} \leqslant 2 \log e(P).
$$
\end{theorem}

Our starting point is the observation that the upper bound is tight if every element of $P$ is incomparable to at most one other element, that is, $P$ is the ordinal sum of one-element and two-elements antichains: $P = A_1 \oplus A_2 \oplus \cdots \oplus A_k$ where each $|A_i| \leqslant 2$. Thus it seems likely that for some small enough constant $\varepsilon > 0$, one can prove that the posets with $|P| \cdot \cent{P} \geqslant (2-\varepsilon) \log e(P)$ possess a very constrained structure. Our main result is to establish such a phenomenon for width-$2$ posets and thus refine Theorem~\ref{thm:factor_2} in this case. We recall that the \emph{width} of poset $P$ is the size of a largest antichain of $P$.

\begin{theorem} \label{thm:width-2} Let $P$ be a width-$2$ poset and let $\kappa_2(P)$ denote the number of size-$2$ connected components of $\incompG{P}$. Then
\begin{equation}
\label{eq:width-2}
|P| \cdot \cent{P} \leqslant (2 - \varepsilon) \log e(P) + \varepsilon \, \kappa_2(P)
\end{equation}
for $\varepsilon = 2 - \frac{3 \log 3 - 2}{\log 3} \simeq 0.26$.
\end{theorem}

Note that Inequality (\ref{eq:width-2}) can be written

$$|P| \cdot \cent{P} \leqslant \left (2 - \varepsilon \left (1 - \frac{\kappa_2(P)}{\log e(P)} \right) \right) \log e(P)$$where $1 - \frac{\kappa_2(P)}{\log e(P)}$ is nonnegative since $e(P) \geqslant 2^{\kappa_2(P)}$ with equality if and only if the components of $\incompG{P}$ are all of size either $1$ or $2$. From this we deduce:

\begin{cor}Let $P$ be a width-$2$ poset, then $|P| \cdot \cent{P} = 2\log(e(P))$ if and only if the maximum degree of $\incompG{P}$ is $1$. \end{cor}

We remark also that upper bounds such as those in Theorems~\ref{thm:factor_11} and \ref{thm:factor_2} translate to upper bounds on the worst case number of comparisons performed by algorithms for a sorting problem known as \emph{sorting under partial information}, see e.g.~\cite{SUPIjournal},\cite{KK} for more details. In the context of this problem, Theorem~\ref{thm:width-2} yields an improvement in the width-$2$ case (\emph{merging under partial information}) because after comparing each of the $\kappa_2(P)$ pairs of elements that form connected components of $\incompG{P}$, the constant in front of $\log e(P)$ decreases from $2$ to $2 - \varepsilon \simeq 1.74$. Furthermore, we point out that the algorithm given by Cardinal et al.~\cite{SUPIjournal} reduces the general problem to the width-$2$ case, hence Theorem~\ref{thm:width-2} also gives an improvement in the general case.

We begin in Section~\ref{sec:graph_entropy} with a brief account of the definitions and main properties of graph entropy. In Section~\ref{sec:poset_entropy}, we specialize this to (in)comparability graphs of posets. In order to help the reader understanding the proof, its general structure is explained in Section \ref{sec:struct_proof}. The intermediate results stated in Section \ref{sec:struct_proof} are then proved in detail in Sections  \ref{sec:struct_GP}, \ref{sec:phantom_edges} and  \ref{sec:small_overlap}.  The final discussion (concluding the proof) is presented in Section \ref{sec:proof}. Finally, Section \ref{sec:special_cases} handles a few particular cases that are not covered by our general argument.


\section{Graph Entropy} \label{sec:graph_entropy}

Here we recall the definition and main properties of the entropy $H(G)$ of a (finite, simple and undirected) graph $G = (V,E)$, as well as the algorithm of K\"orner and Marton to compute $H(G)$ in case $G$ is bipartite. For a more detailed discussion of graph entropy, including the origins of the concept, see the paper of Simonyi \cite{Sim}. Here, we only state the facts that are used in this work.

The definition of $H(G)$ we use relies on the \emph{stable set polytope} 
$$
\STAB(G) := \conv \left(\{\chi^S \in \mathbb{R}^V \mid S \subseteq V,\ S \text{ stable set of } G \} \right)
$$
with $\conv(\cdot)$ denoting the convex hull in $\R^V \cong \mathbb{R}^{|V|}$ and $\chi^S\in \{0,1\}^V$ the characteristic vector of $S$, defined by $\chi^S_v = 1$ if and only if $v \in S$. 

Letting $n := |V|$, the \emph{entropy} of $G$ is defined as 
\begin{equation}
\label{eq:ent_def}
H(G) 
:= \min_{x \in \STAB(G),\, x > 0} - \sum_{v \in V} \frac{1}{n} \log x_v
= \min_{x \in \STAB(G),\, x > 0} \sum_{v \in V} \frac{1}{n} \log \frac{1}{x_v}.
\end{equation}
Note that the function $f(x) := - \sum_{v \in V} \frac{1}{n} \log x_v$ is continuous over $\R^V_{>0}$ and that the point $\frac{1}{n} \chi^V$ is always in $\STAB(G)$, with $f(\frac{1}{n} \chi^V) = \log n$. Thus the minimum in \eqref{eq:ent_def} can be computed over the set $\STAB(G) \cap \{x \in \R^V_{>0} \mid f(x) \leqslant \log n\}$, which is compact. This proves that $H(G)$ is well-defined. Moreover, we have $0 \leqslant H(G) \leqslant \log n$. Finally, since $f(x)$ is strictly convex, its minimizer over $\STAB(G) \cap \R^V_{>0}$ is unique.

We remark that the original definition of graph entropy involves an arbitrary probability distribution on the vertex set $V$ of the graph, whereas the definition used here assumes a uniform distribution. This explains the factor $\frac{1}{n}$ appearing in $H(G)$.

We start with a basic result that enables us to compute the entropy of disconnected graphs. The proof follows directly from the fact that $\STAB(G_1 \cup G_2) = \STAB(G_1) \times \STAB(G_2)$ in case $G_1$ and $G_2$ have disjoint vertex sets.

\begin{proposition}\label{DiG} Let $G_1 = (V_1,\epoch_1)$ and $G_2 = (V_2,\epoch_2)$ be two graphs with disjoint vertex sets and $G = G_1 \cup G_2 = (V_1 \cup V_2, \epoch_1 \cup \epoch_2)$ their disjoint union. Then 
$$
|G| \cdot H(G) =  |G_1| \cdot H(G_1) +  |G_2| \cdot H(G_2).
$$ 
\end{proposition}

For general graphs $G$, no complete linear description of $\STAB(G)$ is known. (In fact, the existence of a tractable description for all graphs $G$ would imply NP $=$ co-NP). Note however that we always have:

$$
\STAB(G) \subseteq \{x \in \mathbb{R}^V_{\geqslant 0} \mid \sum_{v \in K}x_v \leqslant 1 \ \text{ for all cliques } K \text{ of } G \}.
$$

It turns out that the reverse inclusion holds if and only if $G$ is a perfect graph, see Theorem~\ref{thm:Chvatal} below. Recall that a graph $G$ is perfect if $\chi(H) = \omega(H)$ for every induced subgraph $H$ of $G$, where $\omega(H)$ is the size of the largest clique of $H$ and $\chi(H)$ is the chromatic number of $H$. The reader can find more basic information on perfect graphs, e.g., in Diestel \cite{Di}. Later we will use the well-known fact that a graph $G$ is perfect if and only if its complement $\comp{G}$ is perfect.

\begin{theorem}[Chv\'atal \cite{Ch}]\label{thm:Chvatal} A graph $G=(V,E)$ is perfect if and only if 
$$
\STAB(G) = \{x \in \mathbb{R}^V_{\geqslant 0} \mid \sum_{v \in K}x_v \leqslant 1 \ \text{ for all cliques } K \text{ of } G \}.
$$
\end{theorem}

Assume that $G$ is perfect and consider the optimal solution $x^*$ to \eqref{eq:ent_def}. Let $y^*$ be the point with $y^*_v := \frac{1}{nx^*_v}$ for $v \in V$. By optimality of $x^*$, the inequality $\sum_{v \in V} y^*_v x_v \leqslant 1$ is valid for $\STAB(G)$. Then Theorem~\ref{thm:Chvatal} (together with Farkas's lemma) implies that $y^*$ is a convex combination of characteristic vectors of cliques of $G$. Thus $y^* \in \STAB(\comp{G})$. Now, since $x^* \in \STAB(G)$, the inequality $\sum_{v \in V} x^*_v y_v \leqslant 1$ is valid for $\STAB(\comp{G})$. Moreover, this inequality is tight at $y^*$, implying that $y^*$ is a locally optimal solution of \eqref{eq:ent_def} for $\comp{G}$. By convexity, $y^*$ is a globally optimal solution.

This argument implies in particular the following important result due to Csisz\'ar, K\"orner, Lov\'asz, Marton and Simonyi \cite{CKLMS}, which in fact can be turned into a characterization of perfect graphs by considering arbitrary probability distributions supported on $V$, see~\cite{Ko}:

\begin{theorem}[Csisz\'ar et al. \cite{CKLMS}] \label{PGH} For every $n$-vertex perfect graph $G$,
$$
H(G) + H(\comp{G}) = \log n.
$$
\end{theorem}

We will make intensive use of the following theorem of K\"orner and Marton on the entropy of bipartite graphs, and also of the algorithm on which the proof is based. We describe their algorithm after stating the result.

\begin{theorem}[K\"orner and Marton \cite{KoMa}]\label{KM} Let $G$ be a $n$-vertex bipartite graph with bipartition $A \cup B$. Then one can find disjoint subsets $A_1, \ldots, A_k$ and $B_1, \ldots, B_k$ of $A$ and $B$ (respectively) with $A = A_1 \cup \cdots \cup A_k$ and $B = B_1 \cup \cdots \cup B_k$ such that 
\begin{equation}
\label{eq:KM_formula}
H(G) = \sum_{i=1}^k \frac{| A_i | + | B_i |}{n} \, h \left (\frac{| A_i |}{| A_i | + | B_i | } \right )
\end{equation}
with $h : [0,1] \to \mathbb{R}$ defined by $h(x) := -x \log x - (1-x) \log (1-x)$ for $x \in (0,1)$ and $h(0) = h(1) := 0$.
\end{theorem}

In their paper \cite{KoMa}, K\"orner and Marton gave the following algorithm to find pairs $A_i, B_i$ as in Theorem \ref{KM}. For simplicity, we assume first that $G$ has no isolated vertex. Let $A_1$ be a subset of $A$ maximizing the ratio $|A_1|/|B_1|$ where $B_1 := N(A_1)$ is the set of neighbors of $A_1$. Furthermore, choose $A_1$ inclusion-wise minimal with this property. Now iterate this on the graph $G - A_1 - B_1$ with bipartition $(A-A_1) \cup (B-B_1)$ to have the pair $A_2, B_2$, and so on until $A-A_1-\ldots-A_{i-1}$ is empty (in which case $B-B_1-\ldots-B_{i-1}$ is empty, too). 

In case $G$ has isolated vertices, then the first pairs $A_i, B_i$ are of the form $\{a\}, \varnothing$ where $a \in A$ is isolated in $G$, with ratio $|A_i| / |B_i| = +\infty$. The algorithm stops whenever $A-A_1-\ldots-A_{i-1} $ is empty. It may be that $B-B_1-\ldots-B_{i-1}$ is not empty, but then it consists of vertices that are isolated in the initial graph $G$. These are collected in further pairs $A_i, B_i$ of the form $\varnothing, \{b\}$.

We refer to the algorithm described in the two last paragraphs as the \emph{KM algorithm} (for K\"orner and Marton).

\begin{lemma}
\label{lem:KM_connected_pairs}
Let $G$ be a bipartite graph and $A_i, B_i$ for $i = 1, \ldots, k$ denote the pairs constructed by the KM algorithm. Then $G[A_i \cup B_i]$ is connected for all $i$.
\end{lemma}

\begin{proof}
If $G[A_i \cup B_i]$ is not connected, then $A_i$ is the disjoint union of two subsets $A_i^1$ and $A_i^2$ with disjoint neighborhoods $B_i^1$ and $B_i^2$ respectively, in the graph $G-A_1-B_1-\ldots-A_{i-1}-B_{i-1}$. Then
$$
\frac{|A_i|}{|B_i|} = \frac{|A_i^1| + |A_i^2|}{|B_i^1| + |B_i^2|} \leqslant \max \left\{ \frac{|A^1_i|}{|B^1_i|}, \frac{|A^2_i|}{|B^2_i|}\right\},
$$
contradicting the fact that $A_i$ was chosen inclusion-wise minimal among the sets with $|A_i|/|B_i|$ maximum.
\end{proof}

Now, we sketch a proof of Theorem~\ref{KM} based on the KM algorithm. First, consider the point $x^* \in \mathrm{STAB}(G)$ given by 
$$
x^*_u = \frac{|A_i|}{|A_i| + |B_i|} \quad \text{if}\ u \in A_i \quad \text{and} \quad x^*_v = \frac{|B_i|}{|A_i| + |B_i|} \quad \text{if}\ v \in B_i.
$$

Then, represent each vertex of $G$ by a rectangle of width $x^*_v$, height $y^*_v := \frac{1}{nx^*_v}$ and thus area $\frac{1}{n}$. Arrange the $n$ rectangles into a (perfect) packing of the unit square, as illustrated on Figure~\ref{fig:KM_packing}. Since the graph $G$ has no edge from $A_i$ to $B_j$ and $|A_{i}|/|B_{i}| \geqslant |A_{j}|/|B_{j}|$ whenever $i < j$, we have $x^*_u + x^*_v \leqslant 1$ for all $uv \in E$ and hence $x^* \in \STAB(G)$. Proving that $y^* \in \STAB(\comp{G})$ requires a bit more work, but notice that we at least have $\sum_{v \in K} y^*_v \leqslant 1$ for all cliques $K$ of $\comp{G}$ corresponding to rectangles meeting a common vertical. By Theorem~\ref{PGH}, both $x^*$ and $y^*$ are optimal solutions to their respective minimization problems and thus \eqref{eq:KM_formula} holds.

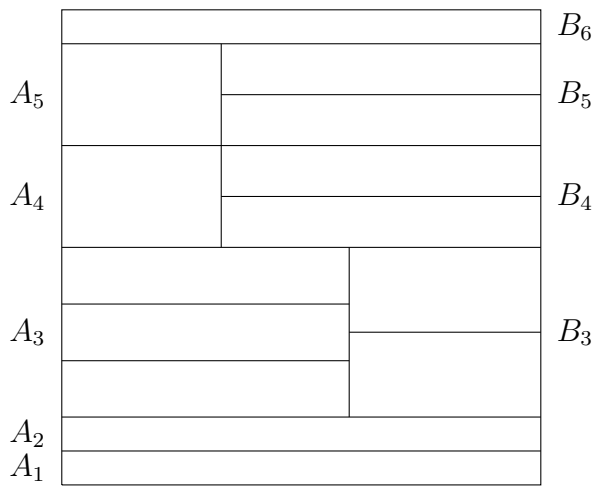
\begin{figure}[ht]
\begin{center}
\begin{tikzpicture}[scale=0.45]
\draw (0,0) -- (14,0) -- (14,14) -- (0,14) -- (0,0);
\draw (0,1) -- (14,1);
\draw (0,2) -- (14,2);
\draw (0,7) -- (14,7);
\draw (0,10) -- (14,10);
\draw (0,13) -- (14,13);
\draw (8.4,2) -- (8.4,7);
\draw (4.66,7) -- (4.66,10);
\draw (4.66,10) -- (4.66,13);
\draw (0,3.66) -- (8.4,3.66); 
\draw (0,5.33) -- (8.4,5.33); 
\draw (8.4,4.5) -- (14,4.5); 
\draw (4.66,8.5) -- (14,8.5); 
\draw (4.66,11.5) -- (14,11.5); 
\node at (-1,0.5) {$A_1$};
\node at (-1,1.5) {$A_2$};
\node at (-1,4.5) {$A_3$};
\node at (-1,8.5) {$A_4$};
\node at (-1,11.5) {$A_5$};
\node at (15,4.5) {$B_3$};
\node at (15,8.5) {$B_4$};
\node at (15,11.5) {$B_5$};
\node at (15,13.5) {$B_6$};
\end{tikzpicture}
\end{center}
\caption{Illustration of the KM algorithm.}
\label{fig:KM_packing}
\end{figure}

\section{Poset Entropy} \label{sec:poset_entropy}

If $P = (X,\leqslant)$ is a finite poset, the entropy of $P$ is defined to be the entropy of its comparability graph $\compG{P}$. We will write this $\ent{P}$. The entropy of the incomparability graph $\incompG{P}$ of $P$ is written $\cent{P}$. 

We insist on the fact that, in this paper, $\ent{P}$ denotes the (K\"orner) entropy of the poset $P$ and not the Shannon entropy of a probability distribution.

Now, we give an equivalent and more intuitive definition of $\ent{P}$ due to Cardinal et al.~\cite{CaSa}. A collection $\{(y_{v^-}, y_{v^{+}})\}_{v \in X}$ of open intervals contained in $(0,1)$ is called \emph{consistent} with $P$ if the associated interval order is an extension of $\leqslant$, that is, if $v < w$ in $P$ implies $y_{v^{+}} \leqslant y_{w^{-}}$ or in other words the interval for $v$ is entirely to the left of the interval for $w$. If $\mathcal{I}(P)$ denotes the set of all these collections of intervals then we have the following result.

\begin{theorem}[Cardinal et al.~\cite{CaSa}]\label{PI} If $P = (X,\leqslant)$ is a poset of order $n$ then
\begin{equation}
\label{eq:ent_intervals}
\ent{P} = \min \left\{-\frac{1}{n} \sum_{v \in X} \log x_v \mid \exists \{(y_{v^-}, y_{v^{+}})\}_{v \in X} \in \mathcal{I}(P) \text{ with } x_v = y_{v^+} - y_{v^-} \forall v \in X\right\}.
\end{equation}
\end{theorem}  

It turns out that not only the lengths $x_v$ of the intervals in an optimal solution to \eqref{eq:ent_intervals} are unique, but also the intervals themselves.

\begin{lemma} \label{lem:unicity_intervals}
The collection of intervals $\{(y^*_{v^-},y^*_{v^+})\}_{v \in X} \in \mathcal{I}(P)$ giving the minimum in \eqref{eq:ent_intervals} is unique.
\end{lemma}

\begin{proof}
Let $x^*_v$ denote the length of the interval for $v \in X$ in any optimal solution to \eqref{eq:ent_intervals}. We know that $x^* \in \STAB(\compG{P})$ and is unique. We have to prove that the lengths $x^*_v$ determine the intervals. To see this define $z^* \in \STAB(\incompG{P})$ by letting $z^*_v = \frac{1}{nx^*_v}$ as in the discussion after Theorem~\ref{thm:Chvatal}. Recall that the inequality $\sum_{v \in X} z^*_v x_v \leqslant 1$ is valid for $\STAB(\compG{P})$ and thus $z^*$ is a convex combination of cliques of $\compG{P}$, that is, of chains of $P$. For each of these chains $C$, we have  $\sum_{v \in C} x^*_v = 1$. In the collection of intervals $\{(y^*_{v^-},y^*_{v^+})\}_{v \in X}$, the chain $C$ is thus formed of consecutive intervals spanning the whole interval $(0,1)$. Therefore we can infer the endpoints of each of the intervals in the chain directly from their lengths. Since the support of $z^*$ is $X$, every element $v$ is contained in such a tight chain $C$. The result follows.
\end{proof}

Following Lemma~\ref{lem:unicity_intervals}, we denote $I(P)$ \emph{the} interval order represented by the optimal collection of intervals for $P$. The collection $\{(y^*_{v^-},y^*_{v^+})\}_{v \in X}$ is called the \emph{canonical} interval representation of $I(P)$.

The following lemma is a direct consequence of the definition of $I(P)$.

\begin{lemma} If $I(P)$ is the interval order represented by the optimal collection of intervals for $P$ then:

\begin{enumerate}[(i)]
\item the poset $I(P)$ is an extension of $P$;
\item the graph  $\incompG{I(P)}$ is a subgraph of $\incompG{P}$;
\item we have  $\cent{P} = \cent{I(P)}$.
\end{enumerate}
\end{lemma}
\begin{proof} 
The first assertion is obvious by definition of $I(P)$. The second one follows from the first one. For the last assertion, let $\{(y^*_{v^-},y^*_{v^+})\}_{v \in X}$ be the canonical interval representation of $I(P)$, where $X$ is the ground set of $P$. Since $I(P)$ is an extension of $P$, we have $\ent{P} \leqslant \ent{I(P)}$. Furthermore, by definition, the collection of intervals $\{(y^*_{v^-},y^*_{v^+})\}_{v \in X}$ gives the optimum in \eqref{eq:ent_intervals} and is at the same time consistent for $I(P)$. Thus $\ent{P} = \ent{I(P)}$
and $\cent{P} = \cent{I(P)}$.
\end{proof}

Hence, to prove Theorem \ref{thm:width-2}, it is tempting to work with $I(P)$ rather than $P$. Indeed, we have $\cent{P} = \cent{I(P)}$ and $\incompG{I(P)}$ has more structure than $\incompG{P}$: for instance, it is an interval graph. However, it turns out that the number of connected components of $\incompG{I(P)}$ and of $\incompG{P}$ may be different, and so $\kappa_2(P) \neq \kappa_2(I(P))$ in general. This we now explain with an example.

\begin{Ex}\label{example_1} Consider the poset $P = (\{a,b,c,d,e,f\}, \leqslant)$ whose incomparability graph is a path on $6$ vertices, see Figure \ref{fig_ex_1}. 

\begin{figure}[ht]
\begin{center}
\begin{tikzpicture}[inner sep=2.5pt]

\filldraw[black] (0,0) circle (1.5pt)
(0,1) circle (1.5 pt) 
(0,2) circle (1.5 pt)
(1,0) circle (1.5 pt) 
(1,1) circle (1.5 pt)
(1,2) circle (1.5 pt);

\draw (0,0) -- (0,2);
\draw (1,0) -- (1,2) -- (0,1);
\draw (0,0) -- (1,1);

\node at (-.2,0) {$a$};
\node at (-.2,1) {$b$};
\node at (-.2,2) {$c$};
\node at (1.2,2) {$f$};
\node at (1.2,1) {$e$};
\node at (1.2,0) {$d$};

\tikzstyle{vtx}=[circle,draw,thick,fill=gray!25]
\node[vtx] (a) at (4,0) {};
\node[vtx] (d) at (5,1) {};
\node[vtx] (b) at (6,0) {};
\node[vtx] (e) at (7,1) {};
\node[vtx] (c) at (8,0) {};
\node[vtx] (f) at (9,1) {};

\draw (a.south) node [below] {$a$};
\draw (b.south) node [below] {$b$};
\draw (c.south) node [below] {$c$};
\draw (d.north) node [above] {$d$};
\draw (e.north) node [above] {$e$};
\draw (f.north) node [above] {$f$};

\node at (.5,-1) {$P$};

\node at (6.5,-1) {$\incompG{P}$};

\draw (a) -- (d) -- (b) -- (e) -- (c) -- (f);

\node[vtx] (a) at (11,0) {};
\node[vtx] (d) at (11,1) {};
\node[vtx] (b) at (12,0) {};
\node[vtx] (e) at (12,1) {};
\node[vtx] (c) at (13,0) {};
\node[vtx] (f) at (13,1) {};

\draw (a) -- (d);
\draw (b) -- (e);
\draw (c) -- (f);

\node at (12,-1) {$\incompG{I(P)}$};

\draw (a.south) node [below] {$a$};
\draw (b.south) node [below] {$b$};
\draw (c.south) node [below] {$c$};
\draw (d.north) node [above] {$d$};
\draw (e.north) node [above] {$e$};
\draw (f.north) node [above] {$f$};

\end{tikzpicture}
\end{center}
\caption{A poset on six elements whose incomparability graph is a path.}
\label{fig_ex_1}
\end{figure}
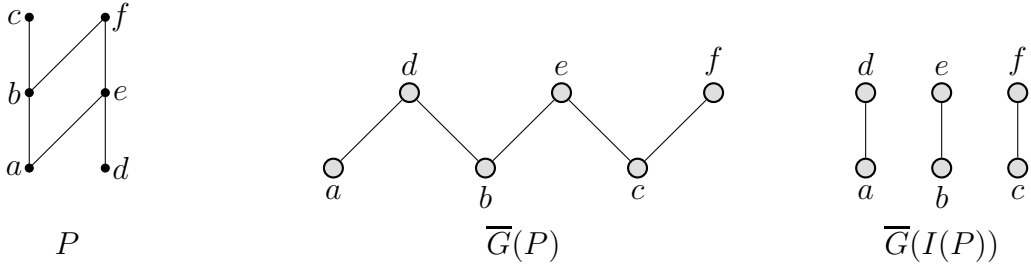

Then $\incompG{P}$ is bipartite with bipartition $A = \{a,b,c\}$, $B = \{d,e,f\}$ and a straightforward application of the KM algorithm gives us $\cent{P} = 1$ with $k=3$, $A_1 = \{a\}$, $B_1 = \{d\}$, $A_2 = \{b\}$, $B_2 = \{e\}$, $A_3 = \{c\}$ and $B_3 = \{f\}$. Notice that Theorem~\ref{thm:width-2} holds in this case because we have $e(P) = 13$ and $\kappa_2(P) = 0$, therefore 
$$
|P| \cdot \cent{P} = 6 = \frac{6}{\log 13} \log 13 \leqslant 1.63 \log 13 \leqslant (2 - \varepsilon) \log e(P).
$$

We now find the graph $\incompG{I(P)}$ and compare it to $\incompG{P}$. Notice first that $$\ent{P} = \log 6 - \cent{P} = \log 3.$$Define now the following collection of intervals contained in $(0,1)$:
\begin{align*}
&(y^*_{a^-},y^*_{a^+}) = (y^*_{d^-},y^*_{d^+}) = (0, 1/3),\\
&(y^*_{b^-},y^*_{b^+}) = (y^*_{e^-},y^*_{e^+}) = (1/3, 2/3),\\
&(y^*_{c^-},y^*_{c^+}) = (y^*_{f^-},y^*_{f^+}) = (2/3, 1).
\end{align*}
Then it is a straighforward task to check that $\{(y_{v^-}, y_{v^{+}}) \mid v \in \{a,b,c,d,e,f\}\}$ is consistent for $P$. Moreover, letting $x^*_v := y^*_{v^+} - y^*_{v^-}$ we have $$-\frac{1}{6} \sum_{v \in \{a,b,c,d,e,f\}} \log x^*_v = \log 3$$ hence we do have the optimal collection of intervals for $P$. The associated graph $\incompG{I(P)}$ consists of three disjoint edges, see Figure \ref{fig_ex_1}. In particular, we see that $\kappa_2(P)  = 0$ and $\kappa_2(I(P)) = 3$.

\end{Ex}

This example shows that it is not possible to work with $I(P)$ directly because some edges in $\incompG{P}$ may disappear in $\incompG{I(P)}$. The next section explains how we can handle this problem.

\section{Structure of the proof of Theorem \ref{thm:width-2}}\label{sec:struct_proof}

The proof of our main theorem being involved, we explain its structure and the intermediate results here. The details will be given in the following sections. 

Our proof is by induction on $n := |P|$. Since the case $n \leqslant 2$ is clear, we assume $n \geqslant 3$. Furthermore, if $\incompG{P}$ is not connected, then $P$ is an ordinal sum $P'_1 \oplus P'_2$ of two smaller posets and we have:
\begin{align*}
|P| \cdot \cent{P} &= |P'_1| \cdot \cent{P'_1} + |P'_2| \cdot \cent{P'_2}
\qquad \text{(by Proposition~\ref{DiG})},\\
\log e(P) &= \log e(P'_1) + \log e(P'_2) \quad \text{and}\\ 
\kappa_2(P) &= \kappa_2(P'_1) + \kappa_2(P'_2).
\end{align*}
By induction, \eqref{eq:width-2} is satisfied by $P'_1$ and $P'_2$, and thus also for $P$.

Hence, we may assume that $\incompG{P}$ is connected. Note that in this case, $\kappa_2(P) = 0$ since $n \geqslant 3$. We study the structure of $\incompG{P}$ closely under the hypothesis $\incompG{P}$ connected and $n\geqslant 3$.

As explained in Section \ref{sec:poset_entropy}, it is tempting to work with $I(P)$ rather than $P$. Example \ref{example_1} shows that this is not really possible because $\incompG{I(P)}$ may be disconnected even if $\incompG{P}$ is connected, hence the number of connected components of size $2$ are not necessarily the same for $\incompG{I(P)}$ and $\incompG{P}$. 

To handle this problem, we will add somes edges between the connected components of $\incompG{I(P)}$. These edges are chosen among those edges of $\incompG{P}$ that disappeared in $\incompG{I(P)}$, we will call them `phantom edges'. The graph $\incompG{I(P)}$ together with the phantom edges is the incomparabilty graph of a width-2 interval order $Q$, and we show that we can assume $P = Q$ for the rest of the proof. These statements concerning the graph $\incompG{P}$ and $\incompG{I(P)}$ are proved carefully in Sections \ref{sec:struct_GP} and \ref{sec:phantom_edges}.


Our strategy now is to seek two elements $u, v$ that are incomparable in $P$ and whose intervals in the canonical interval representation of $I(P)$ have `small' overlap. We will prove that the removal of $uv$ from $\incompG{P}$ yields a new poset $P'$ satisfying the following three conditions:

\begin{enumerate}[({C}1)]
\item \label{c:h_and_e} $\Delta h \leqslant (2 - \varepsilon) \Delta e$ with $\Delta h := n \cent{P} - n \cent{P'} $ and $\Delta e := \log e(P) - \log e(P')$ ;
\item \label{c:ordinal_sum} the poset $P'$ decomposes as an ordinal sum $P'_1 \oplus P'_2$;
\item \label{c:kappa} $\kappa_2(P'_1) = \kappa_2(P'_2) = 0$.
\end{enumerate}

Assuming that such an edge $uv$ can be found, we get
\begin{align*}
|P| \cdot \cent{P} &= |P'| \cdot \cent{P'} + \Delta h \quad \text{(by definition of $\Delta h$)} \\
&= \sum_{i=1,2} |P'_i| \cdot \cent{P'_i} + \Delta h
\quad \text{(since $P' = P'_1 \oplus P'_2$)}\\
&\leqslant 
\sum_{i=1,2} \left((2-\varepsilon) \log e(P'_i) + \varepsilon \underbrace{\kappa_2(P'_i)}_{=0} \right) + \Delta h \quad \text{(by induction)}\\
&\leqslant
\sum_{i=1,2} (2-\varepsilon) \log e(P'_i) + (2-\varepsilon) \Delta e \quad \text{(since $\Delta h \leqslant (2-\varepsilon) \Delta e$)}\\
&\leqslant (2- \varepsilon) (\log e(P') + \Delta e) \quad \text{(since $P' = P'_1 \oplus P'_2$)}\\
&= (2- \varepsilon) \log e(P) + \varepsilon \overbrace{\kappa_2(P)}^{=0} \quad \text{(by definition of $\Delta e$)}.
\end{align*}
and this concludes the proof. Again, the fact that such an edge exists is not obvious, and we prove this in Section \ref{sec:small_overlap}.

The final discussion is presented in Section \ref{sec:proof}. Actually, for a few particular posets, the existence of the edge $uv$ is not guaranteed, and we have to treat these cases by hand. This is done in Section \ref{sec:special_cases}.

\section{The structure of $\incompG{P}$ and $\incompG{I(P)}$} \label{sec:struct_GP}

Since our poset $P = (X,\leqslant)$ has width $2$, we know that $\incompG{P}$ is bipartite with bipartition, say, $A \cup B$. Hence $A$ and $B$ correspond to disjoint chains that cover the poset $P$. Moreover, transitivity of $\leqslant$ implies immediately that for each $u$ in $A$ (respectively in $B$), the neighbors of $u$ in $B$ (respectively in $A$) form a chain in $B$ (in $A$). 

Because $\incompG{P}$ is bipartite, the canonical interval representation of $I(P)$ can be constructed with the KM algorithm. Denote by $z^* \in \STAB(\incompG{P})$ the optimal solution of~\eqref{eq:ent_def} for $\incompG{P}$. Letting $x^*_v := \frac{1}{nz^*_v}$ for $v \in V$, we find the optimal solution of \eqref{eq:ent_def} for $\compG{P}$. Thus the lengths of the intervals are given by:
$$
x^*_u = \frac{|A_i| + |B_i|}{n} \cdot \frac{1}{|A_i|} \quad \text{if}\ u \in A_i \quad \text{and} \quad x^*_v = \frac{|A_i| + |B_i|}{n} \cdot \frac{1}{|B_i|} \quad \text{if}\ v \in B_i.
$$

Notice that we have 
$$
\sum_{u \in A} x^*_u = \sum_{i=1}^k \sum_{u \in A_i} x^*_u = \sum_{i=1}^k |A_i| \left ( \frac{|A_i| + |B_i|}{n} \cdot \frac{1}{|A_i|} \right) = 1
$$
and similarly
$$
\sum_{v \in B} x^*_v=1
$$ 
thus each of the chains $A$ and $B$ yield a chain of consecutive intervals spanning $(0,1)$ in the canonical interval representation of $I(P)$ (unless $A_i = \varnothing$ or $B_i = \varnothing$ for some $i$, that is, unless if $P$ has some cutpoint ---see Figure~\ref{QIP} for an illustration). The endpoints of all the intervals can be directly inferred from this. Moreover, as the following lemma shows, the pairs $A_i, B_i$ are distributed in a very orderly way in the chains $A, B$. Since the result follows directly from Lemma~\ref{lem:KM_connected_pairs} and \cite[Lemma 10]{SUPIjournal}, we omit the proof. For $D$ and $E$ two disjoint subsets of the poset $P$, we write $D \leqslant E$ if $d \leqslant e$ for every $d \in D$ and $e \in E$. Then:

\begin{lemma} \label{LSub} Let $P$ be a width-$2$ poset, let $A_i, B_i$ for $i = 1, \ldots, k$ be the pairs given by the KM algorithm and moreover let $C_i := A_i \cup B_i$ for all $i$. Then there exists a permutation $\sigma$ of $\{1,\ldots,k\}$ such that $C_{\sigma(1)} \leqslant \cdots \leqslant C_{\sigma(k)}$ in $P$. In particular, each $A_i$ and each $B_i$ is an interval in its respective chain.
\end{lemma}

It follows from Lemma~\ref{LSub} that the canonical representation of $I(P)$ has 
\begin{itemize}
\item $|A_{\sigma(i)}|$ consecutive intervals all of length $\frac{|C_{\sigma(i)}|}{n} \cdot \frac{1}{|A_{\sigma(i)}|}$ as well as
\item $|B_{\sigma(i)}|$ consecutive intervals of length $\frac{|C_{\sigma(i)}|}{n} \cdot \frac{1}{|B_{\sigma(i)}|}$
\end{itemize}
within the interval $\left(\sum_{j < i} \frac{|C_{\sigma(j)}|}{n}, \sum_{j \leqslant i} \frac{|C_{\sigma(j)}|}{n}\right)$ for $i = 1, \ldots, k$.

Similarly to Figure~\ref{fig:KM_packing}, we can represent $I(P)$ as a perfect packing of $n$ rectangles of area $\frac{1}{n}$ in the unit square. This time we rotate the packing by 90 degrees and use the linear order on the $C_i$'s induced by $P$. We represent each element $v \in X$ by a rectangle of width $x^*_v$ and height $z^*_v$,  in such a way that the projections of the rectangles on the $x$ axis form the canonical interval representation of $I(P)$, see Figure~\ref{fig:packing_I(P)}.

\begin{figure}[ht]
\begin{center}
\begin{tikzpicture}[scale=0.3,inner sep=2.5pt]
\draw (0,0) -- (22,0) -- (22,22) -- (0,22) -- (0,0);
\draw (8,0) -- (8,22);
\draw (10,0) -- (10,22);
\draw (0,16.5) -- (8,16.5);
\draw (10,7.33) -- (22,7.33);
\draw (2.66, 0) -- (2.66,16.5); 
\draw (5.33,0) -- (5.33,16.5); 

\draw (13,7.33) -- (13, 22); 
\draw (16,0) -- (16,22); 
\draw (19,7.33) -- (19,22); 

\node at (4.5,23) {$A_{\sigma(1)}$};
\node at (8.9,23) {$A_{\sigma(2)}$};
\node at (15.5,23) {$A_{\sigma(3)}$};
\node at (4.5,-1) {$B_{\sigma(1)}$};

\node at (15.5,-1) {$B_{\sigma(3)}$};
\tikzstyle{vtx}=[circle,draw,thick,fill=gray!25]
\node[vtx] (b1) at (1.4,5) {};

\node[vtx] (b2) at (4,5) {};
\node[vtx] (b3) at (6.7,5) {};

\node[vtx] (b4) at (13,5) {};
\node[vtx] (b5) at (19,5) {};

\node[vtx] (a1) at (4,19.5) {};
\node[vtx] (a2) at (9,19.5) {};
\node[vtx] (a3) at (11.5,19.5) {};

\node[vtx] (a4) at (14.5,19.5) {};
\node[vtx] (a5) at (17.5,19.5) {};
\node[vtx] (a6) at (20.5,19.5) {};

\draw [thick] (a1) -- (b1);
\draw [thick] (a1) -- (b2);
\draw [thick] (a1) -- (b3);

\draw [thick] (a3) -- (b4);
\draw [thick] (a4) -- (b4);

\draw [thick] (a5) -- (b5);
\draw [thick] (a6) -- (b5);

\draw [thick, dashed] (a4) -- (b5);

\end{tikzpicture}
\end{center}
\caption{Perfect rectangle packing for $I(P)$. The solid edges are the incomparabilities of $I(P)$. The dashed edge is an incomparability of $P$ that disappeared in $I(P)$.}
\label{fig:packing_I(P)}
\end{figure}
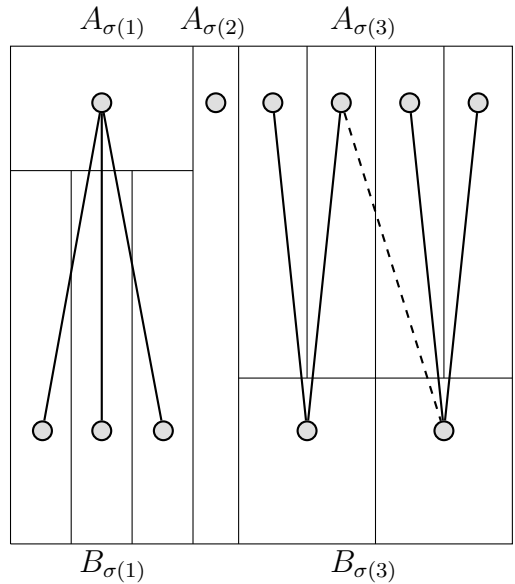

We now study closely the structure of the graph $\incompG{I(P)}$. The connected components of $\incompG{I(P)}$ can actually be inferred directly from the canonical representation of $I(P)$: because the intervals for elements of chain $A$ are consecutive and span the interval $(0,1)$ and similarly for chain $B$, every two consecutive connected components are separated by a \emph{breakpoint}, that is a value $\beta \in [0,1]$ such that every interval $(y^*_{v^-},y^*_{v^+})$ has $\beta \leqslant y^*_{v^-}$ or $y^*_{v^+} \leqslant \beta$, that is, each interval is entirely to the left of or entirely to the right of $\beta$. In particular, $\beta = \sum_{j \leqslant i} \frac{|C_{\sigma(j)}|}{n}$ is a breakpoint for $i = 0, \ldots, k$. Hence $I(P)$ admits at least two breakpoints, 0 and 1, they will be called the \emph{trivial} breakpoints. Let $0 = \beta_0 \leqslant \cdots \leqslant \beta_\ell = 1$ denote the breakpoints of $I(P)$ with $\ell \geqslant k+1$. For $i = 1, \ldots, \ell$, we define the \emph{epoch} $\epoch_i$ to be the set of elements of $P$ represented by the intervals located between $\beta_{i-1}$ and $\beta_i$. Note that in general $\epoch_i$ contains elements from both chains $A$ and $B$. Moreover, since $\sum_{j \leqslant i} \frac{|C_{\sigma(j)}|}{n}$ is a breakpoint for all $i$, each epoch is contained in $C_q$ for some $q \in \{1,\ldots,k\}$.

From now on, we will use the notation $\epochA_i := \epoch_i \cap A$ and $\epochB_i := \epoch_i \cap B$ for $i \in \{1,\ldots,\ell\}$. For the cardinalities, we use $\epochAsize_i = |\epochA_i|$ and $\epochBsize_i = |\epochB_i|$. 

\begin{lemma}\label{lem:CC}
The connected components of $\incompG{I(P)}$ are exactly the subgraphs induced on the epochs $\epoch_i$. Moreover, each of these subgraphs is bipartite with bipartition $\epochA_i \cup \epochB_i$. Finally, we have $\gcd(\epochAsize_i,\epochBsize_i) = 1$. 
\end{lemma}
\begin{proof}
By definition of a breakpoint, $\epoch_i$ is disconnected from $\epoch_j$ for $i \neq j$. Hence it suffices to show that every epoch $\epoch_i$ induces a connected subgraph of $\incompG{I(P)}$. If $|\epoch_i| = 1$ then this is obvious. Assume that $|\epoch_i| \geqslant 2$. Then $\epochAsize_i \geqslant 1$ and $\epochBsize_i \geqslant 1$. In the canonical interval representation of $I(P)$, the intervals for the elements of $\epochA_i$ (respectively $\epochB_i$) are consecutive and span $(\beta_{i-1},\beta_{i})$. Moreover, there is no breakpoint $\beta$ in the open interval $(\beta_{i-1},\beta_{i})$. From this, we conclude that $\epoch_i$ induces a connected component of $\incompG{I(P)}$.

The graph $\incompG{I(P)}$ being itself bipartite with bipartition $A \cup B$, the second assertion is obvious.

For the last assertion, suppose that $\epoch_i$ is contained in $C_j = A_j \cup B_j$. Then we know that the intervals for elements of $\epoch_i$ in the canonical interval representation are:
\begin{itemize}
\item $\epochAsize_i$ consecutive intervals of length $\frac{|C_j|}{n} \cdot \frac{1}{|A_{j}|}$ and  
\item $\epochBsize_i$ consecutive intervals of length $\frac{|C_j|}{n} \cdot \frac{1}{|B_{j}|}$
\end{itemize}
within the interval $(\beta_{i-1}, \beta_i)$. If $\gcd(\epochAsize_i,\epochBsize_i) = t > 1$, then observe that the $\frac{\epochAsize_i}{t}$th interval for an element in $\epochA_i$ and the $\frac{\epochBsize_i}{t}$th interval for an element of $\epochB_i$ have the same right endpoint, which implies the existence of a breakpoint $\beta \in (\beta_{i-1}, \beta_i)$, a contradiction. 
\end{proof}

\section{Phantom edges} \label{sec:phantom_edges} 

We use the same notations as in the previous section. Our goal here is to restore the connectivity of $\incompG{I(P)}$ by adding artificial edges between consecutive epochs ---the `phantom edges'--- so that the incomparability graph of the resulting width-$2$ interval order $Q$ is connected. These edges are chosen among the edges of $\incompG{P}$ that disappeared in $\incompG{I(P)}$, which explains the name `phantom edge'. This implies that $H(P) = H(Q)$ (see Lemma \ref{QI}), which will later allow us to work with $Q$ rather than with $P$.
Since we assume $\incompG{P}$ connected, there is always at least one edge $uv$ between epochs $\epoch_{i}$ and $\epoch_{i+1}$. Moreover:

\begin{lemma}Let $\epoch_{i} = \epochA_i \cup \epochB_i$ and $\epoch_{i+1} = \epochA_{i+1} \cup \epochB_{i+1}$ be two consecutive epochs of $\incompG{I(P)}$. Then there is an edge $uv$ either between either $\epochA_i$ and $\epochB_{i+1}$ or between $\epochB_{i}$ and $\epochA_{i+1}$. Moreover, we may assume either that $u$ is the last element of  $\epochA_i$ and $v$ is the first element of $\epochB_{i+1}$, or $u$ is the last element of $\epochB_{i}$ and $v$ is the first element of $\epochA_{i+1}$. \end{lemma}

\begin{proof} The edge $uv$ between the two epochs exist since we assume that $\incompG{P}$ is connected. Since $\incompG{P}$ is bipartite with bipartition $A \cup B$, we have either $u \in\epochA_i = \epoch_i \cap A$ and $v \in \epochB_{i+1} = \epoch_{i+1} \cap B$ or $u \in \epochB_i = \epoch_i \cap B$ and $v \in \epochA_{i+1} = \epoch_{i+1} \cap A$.

Suppose $u \in \epochA_i$ and $v \in \epochB_{i+1}$, the argument is similar in case $u \in \epochB_{i}$ and $v \in \epochA_{i+1}$. We will show that we can assume that $v$ is the first element of  $\epochB_{i+1}$.

Since the epoch $E_i$ is a connected component of $\incompG{I(P)}$, we know that either $E_i = \{u\}$ or $u$ is adjacent to a vertex $v'$ in $\epochB_{i}$. In the second case, $u$ is adjacent to $v'$ in $\incompG{P}$ also and so $u$ is adjacent to every vertex of the interval $[v',v]$ of the chain $B$. The first element of $\epochA_{i+1}$ being in this interval, we are done.

Suppose then that $E_i = \{u\}$. Let $C_j := A_j \cup B_j$ be the pair given by the KM algorithm and containing $E_i$. By definition of the epochs and the structure of $(A_j,B_j)$, this implies that $A_j = \{u\}$ and $B_j = \varnothing$. But this is a contradiction since we assumed $\incompG{P}$ connected.

Hence we have an edge $uv$ between $u \in \epochA_i$ and $v$ the first element of $\epochB_{i+1}$. Applying the same argument to the element $u$, we can assume that $u$ is the last element of $\epochA_{i+1}$. This concludes the proof. \end{proof}

Notice that in general (that is, unless we both have $\epochAsize_i = \epochBsize_i$ and $\epochAsize_{i+1} = \epochBsize_{i+1}$, which implies $\epochAsize_i = \epochBsize_i = \epochAsize_{i+1} = \epochBsize_{i+1} = 1$ because $\gcd(\epochAsize_i,\epochBsize_i) = \gcd(\epochAsize_{i+1},\epochBsize_{i+1}) = 1$, see Lemma~\ref{lem:CC}), the cases

\begin{enumerate}
\item $u$ is the last element of $\epochA_i$ and $v$ is the first element of $\epochB_{i+1}$,
\item $u$ is the last element of $\epochB_{i}$ and $v$ is the first element of $\epochA_{i+1}$, 
\end{enumerate}
are mutually exclusive. Indeed, since $z^* \in \STAB(\incompG{P})$ we always have $z^*_u + z^*_v \leqslant 1$. To obtain $Q$ from $I(P)$, we add one such edge $uv$ to the incomparability graph of $I(P)$ for each $i \in \{1, \ldots, \ell-1\}$. We call these extra edges \emph{phantom edges}. 

\begin{Ex}
Consider the poset $P$ of Example \ref{example_1}. Then the phantom edges are exactly $db$ and $ec$. Hence in this example we have $Q = P$. This is not always the case: the reader can check this if $P$ is the disjoint union of two chains of size $2$. In that case, $\incompG{P}$ is a complete bipartite graph on $2 + 2$ vertices, $\incompG{I(P)}$ is a perfect matching on $4$ vertices and $\incompG{Q}$ is a path with $4$ vertices.
\end{Ex}

\begin{lemma}
\label{QI} The poset $Q$ satisfies the following conditions:
\begin{enumerate}[(i)]
\item $\incompG{Q}$ is connected; 
\item $Q$ is a width-$2$ interval order;
\item $\ent{Q} = \ent{P}$;
\item $e(Q) \leqslant e(P)$.
\end{enumerate}
\end{lemma}
\begin{proof}
(i) This follows from Lemma~\ref{lem:CC} and the construction of $Q$.

(ii) The fact that the width of $Q$ is $2$ follows from the assumption that the width of $P$ is $2$ and from the fact that $Q$ is an extension of $P$.

Now we explain how to modify the canonical representation of $I(P)$ in order to obtain an interval representation of $Q$. As before, let $\ell$ denote the number of epochs $\epoch_i$. Thus $q = \ell - 1$ gives the number of breakpoints in $(0,1)$. For each breakpoint $\beta \in (0,1)$ we introduce a gap of $1/q$ between the intervals on each side of $\beta$, so that all intervals in the representation now fit in the interval $(0,2)$, and cover half of its area.

Consider some breakpoint $\beta$ that has a corresponding phantom edge $uv$ with the interval for $u$ touching the left of the newly created gap and the interval for $v$ touching the right of that gap. Then by adding $1/q$ to the right endpoint of the interval for $u$ and subtracting $1/q$ to the left endpoint of the interval for $v$, we make sure that these intervals intersect. After having treated in such a way all breakpoints that have a phantom edge, we obtain an interval representation for the poset $Q$. This is illustrated in Figure~\ref{QIP}.

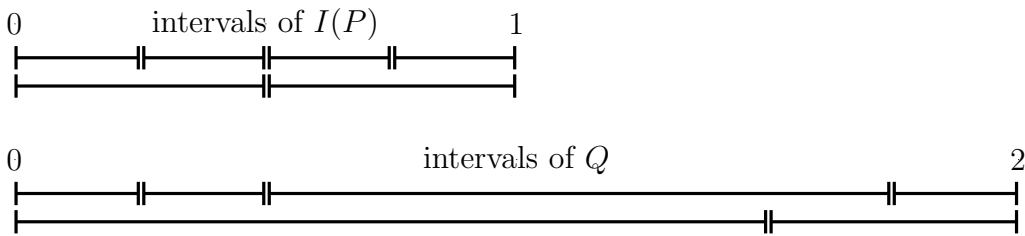
\begin{figure}[ht]
\begin{center}
\begin{tikzpicture}[scale=0.3]

\draw (0,1.5)--(0,1.5) node {$\small 0$};
\draw (22,1.5)--(22,1.5) node {$\small 1$};

\draw[very thick, |-|] (0,0)--(5.5,0);
\draw[very thick, |-|] (5.6,0)--(11,0);
\draw[very thick, |-|] (11.1, 0)--(16.5,0);
\draw[very thick, |-|] (16.6,0)--(22,0);

\draw[very thick, |-|] (0,-1.25) -- (11,-1.25);
\draw[very thick, |-|] (11.1,-1.25) -- (22,-1.25);

\draw (11,1.5)--(11,1.5) node {intervals of $I(P)$} ;

\draw (22,-4.5)--(22,-4.5) node {intervals of $Q$} ;

\draw (0,-4.5)--(0,-4.5) node {$\small 0$};
\draw (44,-4.5)--(44,-4.5) node {$\small 2$};

\draw[very thick, |-|] (0,-6)--(5.5,-6);
\draw[very thick, |-|] (5.6,-6)--(11,-6);

\draw[very thick, |-|] (38.5,-6)--(44,-6);
\draw[very thick, |-|] (11.1,-6)--(38.4,-6);

\draw[very thick, |-|] (0,-7.25) -- (33,-7.25);
\draw[very thick, |-|] (33.1,-7.25) -- (44,-7.25);

\end{tikzpicture}
\end{center}
\caption{The interval representation of $I(P)$ and $Q$ for $P$ the disjoint union of a chain of length $4$ and a chain of lentgh $2$.}
\label{QIP}
\end{figure}

(iii) The poset $I(P)$ is an extension of $Q$ which is in turn an extension of $P$. Hence $\ent{P} \leqslant \ent{Q} \leqslant \ent{I(P)}$. But we know $\ent{I(P)} = \ent{P}$, so we have equality throughout. 

(iv) Obviously, $e(Q) \leqslant e(P)$ since $Q$ extends $P$.
\end{proof}

Now assume that \eqref{eq:width-2} holds for $Q$. Then, by Lemma~\ref{QI}, we get
\begin{align*}
|P| \cdot \cent{P} &= |Q| \cdot \cent{Q} \quad \text{(because $\ent{P} = \ent{Q}$)}\\
&\leqslant (2 - \varepsilon) \log e(Q)
\quad \text{(because \eqref{eq:width-2} holds for $Q$ and $\kappa_2(Q) = 0$)}\\
&\leqslant (2 - \varepsilon) \log e(P)
\quad \text{(because $e(Q) \leqslant e(P)$).}
\end{align*}
Therefore, to prove Theorem~\ref{thm:width-2}, we can assume that $P = Q$, that is, $P$ is a width-$2$ interval order that coincides with $I(P)$ except perhaps for a few incomparabilities. 

\section{Removing an incomparability with a small overlap} \label{sec:small_overlap}
 
As discussed in Section \ref{sec:struct_proof}, to conclude the proof of Theorem~\ref{thm:width-2}, we should now prove the existence of an edge $uv$ in $\incompG{P}$ such that its removal yields a new poset $P'$ satisfying the conditions (C\ref{c:h_and_e}), (C\ref{c:ordinal_sum}) and  (C\ref{c:kappa}). Recall also that we may assume the following facts on the width-2 poset $P$:
\begin{itemize}
\item it has $n \geqslant 3$ elements,
\item its incomparability graph is connected (hence $\kappa_2(P) = 0$),
\item finally, $P$ coincides with $I(P)$ except for a few pairs of elements: the phantom edges.
\end{itemize}


In particular, $\incompG{P}$ has no isolated vertex and thus we have $\epochAsize_i \geqslant 1$ and $\epochBsize_i \geqslant 1$ for all $i$.

\subsection{Removing a phantom edge} 

It turns out that, except in a few particular cases, if $\incompG{P}$ admits phantom edges, then the conditions here above are easily satisfied. Indeed, if the edge $uv$ is a phantom edge, we have $\Delta h = 0$. In particular, (C\ref{c:h_and_e}) holds. Moreover, (C\ref{c:ordinal_sum}) also holds because the removal of the incomparability $uv$ disconnects $\incompG{P}$ into exactly two connected components. Thus the only condition that remains to be checked is (C\ref{c:kappa}). This condition always holds unless $uv$ links the first pair of epochs $\epoch_1, \epoch_2$ and $|\epoch_1| = 2$ or $uv$ links the last pair of epochs $\epoch_{\ell-1}, \epoch_{\ell}$ and $|\epoch_{\ell}| = 2$. Hence a good choice of $uv$ is possible whenever $\ell \geqslant 4$. In case $2 \leqslant \ell \leqslant 3$, there exists a good phantom edge unless $(|\epoch_{1}|,\ldots,|\epoch_{\ell}|)$ is equal to $(2,m)$ or $(m,2)$ or $(2,m,2)$ for some integer $m \geqslant 2$. In the case $\ell=1$ and in these cases, taking $uv$ to be a phantom edge will not work and we have to choose $uv$ differently.

\subsection{Removing an edge within an epoch} \label{sec:removC}

Fix an index $i \in \{1,\ldots,\ell\}$. Now, we inspect more closely the structure of the subposet of $P$ induced on $\epoch_i = \epochA_i \cup \epochB_i$. We denote this subposet by $P_i$. Without loss of generality, we assume that $\epochAsize_i  \geqslant \epochBsize_i  \geqslant 1$. Since we assumed that $P$ coincides with $I(P)$ (except for the phantom edges), the subposet $P_i$ agrees with the subposet of $I(P)$ induced on $\epoch_i$, and is thus an interval order that admits an interval representation in $(0,1)$ obtained as follows:
\begin{itemize}
\item starting from $0$, put side by side $\epochAsize_i$ intervals of length $\frac{1}{\epochAsize_i}$;
\item starting again from $0$, put side by side $\epochBsize_i$ intervals of length $\frac{1}{\epochBsize_i}$.
\end{itemize}
Recall that in the canonical interval representation of $I(P)$, the corresponding intervals have length $\frac{|C_j|}{n} \cdot \frac{1}{|A_j|} \propto \frac{1}{|\epochA_i|}$ and $\frac{|C_j|}{n} \cdot \frac{1}{|B_j|} \propto \frac{1}{|\epochB_i|}$ respectively, where $j \in \{1,\ldots,k\}$ is such that $\epoch_i \subseteq C_j$ and the proportionality constants are identical. In the above representation, we delete all intervals for elements not in $P_i$ and then rescale (and translate) so that the intervals again span $(0,1)$. 

By Lemma~\ref{lem:CC}, we know that $\gcd(\epochAsize_i,\epochBsize_i) = 1$.

\begin{lemma} \label{AB} If $\epochAsize_i \geqslant 2$ and $\epochBsize_i \geqslant 2$, there exist two elements $u, v \in P_i$ such that the corresponding intervals overlap in an interval of length exactly $\frac{1}{\epochAsize_i \epochBsize_i}$.
\end{lemma}

\begin{proof}
It suffices to show that there are two integers $m$ and $p$ with $0 < m < \epochAsize_i$, $0 < p < \epochBsize_i$ and $\left| \frac{m}{\epochAsize_i } - \frac{p}{\epochBsize_i} \right| = \frac{1}{\epochAsize_i \epochBsize_i}$, that is, $|m \epochBsize_i - p \epochAsize_i | = 1$. Since $\gcd(\epochAsize_i ,\epochBsize_i) = 1$ there exist integers $m, p$ with $| m \epochBsize_i - p \epochAsize_i  | = 1$. It remains to prove that we can assume $0 < m < \epochAsize_i $ and $0 < p < \epochBsize_i$. Note that $| m \epochBsize_i - p \epochAsize_i  | = 1$ implies $|(m - t\epochAsize_i )\epochBsize_i - (p - t\epochBsize_i)\epochAsize_i | = 1$ for every $t \in \mathbb{Z}$. Hence one may suppose $0 < m \leqslant \epochAsize_i $ and this implies $0 < p \leqslant \epochBsize_i$. But $m = \epochAsize_i $ implies $\epochAsize_i  = 1$, and $p = \epochBsize_i$ implies $\epochBsize_i = 1$. This concludes the proof.
\end{proof}

In fact we can always suppose that there exist $m$ and $p$ with%
\begin{equation}
\label{eq:overlap}
\frac{m}{\epochAsize_i } - \frac{p}{\epochBsize_i} = \frac{1}{\epochAsize_i \epochBsize_i}.
\end{equation}
Indeed, if $\frac{m}{\epochAsize_i} - \frac{p}{\epochBsize_i} = -\frac{1}{\epochAsize_i\epochBsize_i}$ we just remplace $m$ by $\epochAsize_i-m$ and $p$ by $\epochBsize_i-p$. Hence we know that the corresponding intervals are the $m$-th of length $1/\epochAsize_i$ and the $(p+1)$-th of length $1/\epochBsize_i$. In this case, an interval of length $1/\epochAsize_i$ immediately to the right of the interval for $u$ must exist (the associated element of $P_i$ is written $u'$), as well as an interval of length $1/\epochBsize_i$ immediately to the left of the interval for $v$ (the associated element of $P_i$ is written $v'$),  see Figure~\ref{FigInt1}. In the figure and henceforth, we denote $I(u)$ the interval for $u$, and similarly for the other elements. 

\begin{figure}[ht]
\begin{center}
\begin{tikzpicture}[scale=1.5]

\draw[very thick, |-|] (0,0)--(.98,0);
\draw[very thick, |-|] (1.02,0)--(2.1,0);
\draw[very thick, |-|] (.82,.4) -- (2.2,.4);
\draw[very thick, |-|] (-.78,.4) -- (.78,.4);

\node at (1.5,0.8) {$I(v)$};
\node at (0,0.8) {$I(v')$};

\node at (.4,-.4) {$I(u)$};
\node at (1.6,-.4) {$I(u')$};

\draw[dashed] (.8,1.0) -- (.8,-0.7);
\draw[dashed] (1,1.0) -- (1,-0.7);

\draw (.9,-1.1) -- (.9,-1.1) node {\small $1/\epochAsize_i\epochBsize_i$};
\end{tikzpicture}
\end{center}
\caption{The intervals $I(u)$ and $I(v)$ of Lemma~\ref{AB} and the neighboring intervals.}
\label{FigInt1}
\end{figure}
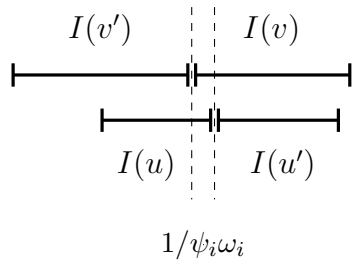

\begin{lemma}\label{EE}
Let $P_i$ be the subposet of $P$ induced by some epoch $\epoch_i$ with $\epochAsize_i \geqslant \epochBsize_i \geqslant 2$ and $u$, $v$ be two elements of $P_i$ whose intervals in the interval representation of $P_i$ are such that $I(u) \cap I(v)$ is of length $1/\epochAsize_i\epochBsize_i$. Then the suppression of $uv$ from $\incompG{P}$ yields a poset $P'$ with $$\Delta h := n \cent P - n\cent {P'} \leqslant  2 \log \left (\frac{1}{1 - \frac{1}{(\epochAsize_i+\epochBsize_i)^2}} \right)$$and $P'$ is an ordinal sum of two smaller posets $P'_1$ and $P'_2$. Moreover, unless $\epochAsize_i = \epochBsize_i + 1$, both $P'_1$ and $P'_2$ have at least three elements that are also in $P_i$.
\end{lemma}

\begin{proof} 
Let $n := | P |$ $:= | X |$ with $X$ the ground set of poset $P$. As noticed above, we can assume $I(u)$ and $I(v)$ are such that the left endpoint of $I(u)$ is to the left of $I(v)$, as in Figure~\ref{FigInt1}. As before, $u'$ is the element of $P$ such that $I(u')$ follows $I(u)$ and $v'$ is the element of $P$ such that $I(v)$ follows $I(v')$, see Figure \ref{FigInt1}. By this local modification we get a new poset $P'$ with $\incompG{P'} = \incompG{P} - uv$. 

The idea is to move the right endpoint of $I(u)$, which is also the left endpoint of $I(u')$, by $\frac{1}{\epochAsize_i(\epochAsize_i+\epochBsize_i)} = \frac{\epochBsize_i}{(\epochAsize_i\epochBsize_i)(\epochAsize_i+\epochBsize_i)}$ to the left and the left endpoint of $I(v)$, which is also the right endpoint of $I(v')$, by $\frac{1}{\epochBsize_i(\epochAsize_i+\epochBsize_i)}$ to the right, see Figure~\ref{FigInt3}.

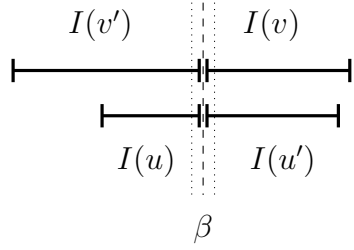
\begin{figure}[ht]
\begin{center}
\begin{tikzpicture}[scale=1.5]

\draw[very thick, |-|] (0,0)--(.88,0);
\draw[very thick, |-|] (.92,0)--(2.1,0);
\draw[very thick, |-|] (.92,.4) -- (2.2,.4);
\draw[very thick, |-|] (-.78,.4) -- (.88,.4);

\node at (1.5,0.8) {$I(v)$};
\node at (0,0.8) {$I(v')$};

\node at (.4,-.4) {$I(u)$};
\node at (1.6,-.4) {$I(u')$};

\draw[dotted] (.8,1.0) -- (.8,-0.7);
\draw[dashed] (.9,1.0) -- (.9,-0.7);
\draw[dotted] (1,1.0) -- (1,-0.7);
\node at (.9,-1) {$\beta$};
\end{tikzpicture}
\end{center}
\caption{After the modification of the intervals $I(u)$, $I(v)$, $I(u')$ and $I(v')$.}
\label{FigInt3}
\end{figure}

We denote as before $x_w$ the length of the interval $I(w)$ for $w \in X$ in $I(P)$, and $\tilde{x}_w$ the length of that interval after modification. Since $I(u)$ and $I(u')$ have length $1/\epochAsize_i$ and $I(v)$, $I(v')$ have length $1/\epochBsize_i$ we have: 
$$
\tilde{x}_u = \frac{1}{\epochAsize_i} \left ( 1 - \frac{1}{\epochAsize_i+\epochBsize_i} \right) \quad \text{and} \quad \tilde{x}_{u'} = \frac{1}{\epochAsize_i} \left ( 1 + \frac{1}{\epochAsize_i+\epochBsize_i} \right)
$$
and also, for the elements in the other chain,
$$
\tilde{x}_{v} = \frac{1}{\epochBsize_i} \left ( 1 - \frac{1}{\epochAsize_i+\epochBsize_i} \right) \quad \text{and} \quad \tilde{x}_{v'} = \frac{1}{\epochBsize_i} \left ( 1 + \frac{1}{\epochAsize_i+\epochBsize_i} \right).
$$

By \eqref{eq:ent_intervals}, this shows
\begin{align*}
n \cdot H(P') \leqslant  - \sum_{w \in X} \log \tilde{x}_w &=  -\sum_{w \in X} \log x_w + 2 \log \left (\frac{1}{\epochAsize_i} \right) +  2\log \left(\frac{1}{\epochBsize_i}\right)\\
 &\mbox{} \quad -  \log \frac{1}{\epochAsize_i} \left ( 1 - \frac{1}{\epochAsize_i+\epochBsize_i} \right) -  \log \frac{1}{\epochBsize_i} \left ( 1 - \frac{1}{\epochAsize_i+\epochBsize_i} \right)\\
 &\mbox{} \quad -\log \frac{1}{\epochAsize_i} \left ( 1 + \frac{1}{\epochAsize_i+\epochBsize_i} \right) -  \log \frac{1}{\epochBsize_i} \left ( 1 + \frac{1}{\epochAsize_i+\epochBsize_i} \right)\\
 &= -\sum_{w \in X} \log x_w + \underbrace{2 \log \left( \frac{1}{1 - \frac{1}{(\epochAsize_i+\epochBsize_i)^2}} \right)}_{=: f(\epochAsize_i,\epochBsize_i)}.
\end{align*}
So $n \cdot H(P')  \leqslant n \cdot H(P)  + f(\epochAsize_i,\epochBsize_i)$, which by theorem
\ref{PGH} implies $\Delta h \leqslant f(\epochAsize_i,\epochBsize_i)$.

Now by the structure of the intervals in $I(P)$ (Lemma~\ref{LSub}) it is clear that $P' = P'_1 \oplus P'_2$ for two smaller posets $P'_1$ and $P'_2$: the elements of $P'_1$ are those whose new interval is to the left of the breakpoint $\beta$ created by the local modification (see Figure~\ref{FigInt3}), and similarly the elements of $P'_2$ are those whose new interval is to the right of $\beta$. It is clear that both $P'_1$ and $P'_2$ each contain at least two elements of $P_i$, namely, $u$ and $v'$ for $P'_1$ and $v$ and $u'$ for $P'_2$. 

If $P'_1$ has less than $3$ elements of $P_i$, then $I(u)$ is the first interval of $P_i$ having length $1/\epochAsize_i$ and $I(v)$ the second interval of $P_i$ having length $1/\epochBsize_i$. This implies $\ell=1$ and $m+1 = 2$ in \eqref{eq:overlap} that is $\epochBsize_i = \epochAsize_i+1$, and this is a contradiction since we supposed $\epochAsize_i \geqslant \epochBsize_i$. Similarly if $P'_2$ has less than $3$ elements of $P_i$ then $\ell=\epochAsize_i-1$ and $m+1=\epochBsize_i$ and this implies $\epochAsize_i = \epochBsize_i + 1$.
\end{proof}

Now, we analyze how the number of linear extensions of $P$ changes after the deletion of the incomparability $uv$.

\begin{lemma} \label{EAB} Let $P_i$ be the subposet of $P$ induced by some $\epoch_i$ and $u$, $v$ be two elements of $P_i$ whose intervals in the interval representation of $P_i$ are such that $I(u) \cap I(v)$ is of length $1/\epochAsize_i\epochBsize_i$. Let $P' = P'_1 \oplus P'_2$ be the poset obtained by deleting the edge $uv$ from $\incompG{P}$. Then $$\Delta e := \log e(P) - \log e(P') \geqslant \log \left ( 1 + \frac{1}{2 \frac{\epochAsize_i}{\epochBsize_i} + 4} \right ).$$ \end{lemma}

\begin{proof}
The inequality we have to prove can be rewritten
$$
e(P) \geqslant e(P'_1 \oplus P'_2) \cdot \left (1 + \frac{1}{2 \frac{\epochAsize_i}{\epochBsize_i}+4} \right ).
$$
Since the linear extensions of $P'_1 \oplus P'_2$ correspond to the linear extensions of $P$ with $u \prec v$, we have to establish that a big enough fraction of the linear extensions $\prec$ of $P$ have $v \prec u$. 

We call a linear extension $\prec$ of $P$ \emph{backward} if $v \prec u$, and \emph{forward} if $u \prec v$. The forward extensions correspond to those of $P' =  P'_1 \oplus P'_2$. Clearly, for a backward extension of $P$ we have in particular:

\begin{quote}
(*) $v \prec w$ for every element $w \neq u$ incomparable to $v$, and $z \prec u$ for every element $z \neq v$ incomparable to $u$.
\end{quote}

Indeed, for such a $w$, the interval $I(w)$ is located to the right of $I(u)$ hence $u \leqslant w$ in $P_i$ and by transitivity $v \prec u \prec w$. The second part of the statement is proved similarly.

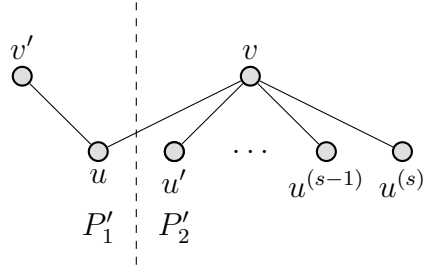
\begin{figure}[ht]
\begin{center}
\begin{tikzpicture}[inner sep = 2.5pt]
\tikzstyle{vtx}=[circle,draw,thick,fill=gray!25]
\node[vtx] (v') at (0,1) {};
\node[vtx] (u) at (1,0) {};
\node[vtx] (v) at (3,1) {};
\node[vtx] (u') at (2,0) {};
\node[vtx] (u'') at (2,0) {};
\node[vtx] (u^s-1) at (4,0) {};
\node[vtx] (u^s) at (5,0) {};
\draw (v'.north) node[above] {$v'$};
\draw (v.north) node[above] {$v$};
\draw (u.south) node[below] {$u$};
\draw (u'.south) node[below] {$u'$};
\draw (u^s-1.south) node[below] {$u^{(s-1)}$};
\draw (u^s.south) node[below] {$u^{(s)}$};
\node at (3,0) {$\ldots$};

\draw (v') -- (u) -- (v) -- (u');
\draw (v) -- (u^s-1);
\draw (v) -- (u^s);

\draw[dashed] (1.5,-1.5) -- (1.5,2);
\node at (1,-1) {$P'_1$};
\node at (2,-1) {$P'_2$};
\end{tikzpicture}
\end{center}
\caption{The local structure of the graph $\incompG{P_i}$.}
\label{StruIJ}
\end{figure}

We call a forward extension \emph{good} if it satisfies property (*). Note that any good forward extension gives one backward extension, simply by interchanging $u$ and $v$ (which are consecutive in any good forward extension).

Every linear extension $\prec$ of $P$ induces an orientation of the incomparability graph $\incompG{P}$: we orient each edge $wz$ from $w$ to $z$ if $w \prec z$ in the extension. We define an equivalence relation $\sim$ on the set $\mathcal{E}(P)$ of linear extensions of $P$ by letting $\prec_1 \sim \prec_2$ if and only if $\prec_1$ and $\prec_2$ induce the same orientation of the edges $\incompG{P}$ incident to neither $u$ nor $v$.

Each class of this equivalence relation $\sim$ contains precisely:

\begin{itemize}
\item one good forward extension,
\item one good backward extension,
\item possibly some more forward extensions that are not good.
\end{itemize}

Hence the number of backward extensions is exactly the number of good forward extensions, and this is the number of classes of $\sim$, this quantity being at least
$$
\frac{e(P'_1 \oplus P'_2)}{M}
$$
where $M$ is the maximum cardinality of one class of $\sim$. Hence, summing the total number of forward extension and the minimum number of backward extension we have
$$
e(P) \geqslant e(P'_1 \oplus P'_2) + \frac{e(P'_1 \oplus P'_2)}{M} =  e(P'_1 \oplus P'_2) \cdot \left (1 + \frac{1}{M} \right)
$$
and it remains to prove
\begin{equation}
\label{eq:UB_on_M}
M \leqslant 2 \frac{\epochAsize_i}{\epochBsize_i} + 4.
\end{equation}
To do so we upper bound, for any given forward extension $\prec$, the number of possible orientations for the edges of $\incompG{P}$ that are incident to $u$ or $v$.

Let $$u = u^{(0)} < u' = u^{(1)} < \ldots < u^{(s)}$$denote the neighbors of $v$ in $\incompG{P}$ and $v' < v$ denote the neighbor of $u$ in this graph, see Figure~\ref{StruIJ}. Note that $v$ and $v'$ are the only neighbors of $u$ in $\incompG{P}$ because the interval $I(u)$ has length $\frac{1}{\epochAsize_i}$, the intervals $I(v)$ and $I(v')$ have length $\frac{1}{\epochBsize_i}$ and $\epochAsize_i \geqslant \epochBsize_i$ by assumption.

Looking at the interval representation of $P_i$, we see that the intervals $I(u')$, \ldots, $I(u^{(s-1)})$ are all included in $I(v)$ and cover an area that is at most the area of $I(v)$. In other words, we have
$$
\frac{s-1}{\epochAsize_i} \leqslant \frac{1}{\epochBsize_i}
\iff 
s \leqslant \frac{\epochAsize_i}{\epochBsize_i} + 1.
$$
We have exaclty $s+1$ different possibilities for inserting $v$ in the opposite chain, and hence a forward extension can orient the edges of $\incompG{P}$ incident to $u$ in exactly $s+1$ ways (recall that $u \prec v$ because the extension is forward). Because the edge $uv'$, which is the last edge we have to consider, can be oriented in at most two ways, we get $M \leqslant 2(s+1) \leqslant 2(\frac{\epochAsize_i}{\epochBsize_i} + 2)$ and \eqref{eq:UB_on_M} follows.
\end{proof}

And finally:

\begin{lemma} \label{LF} For all $x \geqslant y \geqslant 2$, we have
$$
2 \log \left (\frac{1}{1 - \frac{1}{(x+y)^2}} \right) \leqslant 
\frac{3}{2} \log \left ( 1 + \frac{1}{2 \frac{x}{y} + 4} \right).
$$
\end{lemma}

\begin{proof}
First, note that 
$$ 
2\log \left (\frac{1}{1 - \frac{1}{(x+y)^2}} \right) = 2\log \left( \frac{\frac{x}{y} + \frac{y}{x} + 2 }{\frac{x}{y} + \frac{y}{x} + 2 - \frac{1}{xy}} \right)
$$
and since $x, y \geqslant 2$, we have $2 - \frac{1}{xy} \geqslant \frac{7}{4}$. Hence letting $u := \frac{x}{y}$ we have 
$$
2\log \left (\frac{1}{1 - \frac{1}{(x+y)^2}} \right) \leqslant  2\log \left(\frac{u + \frac{1}{u} + 2}{u + \frac{1}{u} + \frac{7}{4}} \right)
\leqslant
\log \left(\frac{u + 2}{u + \frac{7}{4}} \right)
$$
and
$$
\log \left ( 1 + \frac{1}{2 \frac{x}{y} + 4} \right) =  \log \left( \frac{u+5/2}{u+2} \right).
$$

The target inequality is thus implied by
$$
(u+2)^7 \leqslant (u+7/4)^4(u+5/2)^3
$$
which can be rewritten (after performing a straighforward computation) as
$$
0 \leqslant \frac{1}{2}u^6 + \frac{45}{8}u^5 + \frac{209}{8}u^4 + \frac{16401}{256}u^3 + \frac{44751}{512}u^2 + \frac{64323}{1024}u + \frac{37981}{2048}.
$$ 
The result follows.
\end{proof}

\section{The final discussion} \label{sec:proof}

\begin{proof}[Proof of Theorem~\ref{thm:width-2}]
Let now $P$ be any width-$2$ poset. The proof is by induction on $n$. Clearly, we may assume $n \geqslant 3$ since the theorem holds for $n \leqslant 2$. We have established in Section~\ref{sec:phantom_edges} that we may without loss of generality assume $P$ is an interval order that coincides with $I(P)$ except perhaps for a few incomparabilities: the phantom edges. 

Let $\epoch_{1} = \epochA_1 \cup \epochB_1,$ $\ldots$ $,\epoch_{l} = \epochA_\ell \cup \epochB_\ell$ be the epochs of $P$. In Section~\ref{sec:small_overlap} we proved that we may assume the following:

\begin{itemize}
\item for each $i$ we have $\gcd(\epochAsize_i,\epochBsize_i) = 1$ where $\epochAsize_i := |\epochA_i|$ and $\epochBsize_i := |\epochB_i|$;
\item either $\ell =1$, or $\ell = 2$ and $(|\epoch_{1}|, |\epoch_{2}|)$ is equal to $(2,m)$ or $(m,2)$ for an integer $m \geqslant 2$, or $\ell = 3$ and $(|\epoch_{1}|, |\epoch_{2}| ,|\epoch_{3}|)$ is equal $(2,m,2)$ for an integer $m \geqslant 2$.
\end{itemize}

Let $i \in \{1,2,3\}$ be such that $|\epochA_i|$ is maximum and assume without loss of generality that $\epochAsize_i \geqslant \epochBsize_i$. Then, combining Lemmas~\ref{AB}, \ref{EE}, \ref{EAB} and \ref{LF}, we are able to find a good edge to remove from $\incompG{P}$ in case $\epochAsize_i, \epochBsize_i \geqslant 2$ and $\epochAsize_i > \epochBsize_i + 1$. With this good edge in hand, we can complete the proof as explained in Section~\ref{sec:struct_proof}. Hence the only cases left to consider are the following ones:

\begin{enumerate}
\item $\ell =1$, $\epochAsize_1 = \epochBsize_1 + 1$;
\item $\ell =1$, $\epochBsize_1 = 1$ and $\epochAsize_1 \geqslant 3$;
\item $\ell =2$, $\epochBsize_i = 1$ and $\epochAsize_i \geqslant 2$;
\item $\ell =3$, $\epochBsize_i = 1$ and $\epochAsize_i \geqslant 2$;
\item $\ell =2$, $\epochAsize_i = \epochBsize_i + 1$.
\item $\ell =3$, $\epochAsize_i = \epochBsize_i + 1$.
\end{enumerate}

Note that for the second case, we assume $\epochAsize_1 \geqslant 3$ because the first one encompass the possibility $\epochAsize_1 = 2$ and $\epochBsize_1 = 1$. Each of these cases follows from the results of Section~\ref{sec:special_cases} below. In particular, we prove that $\ell = 1$ and $\epochAsize_1 = \epochBsize_1 + 1$ implies that $\incompG{P_1}$ is a path and that the theorem holds in this case. This concludes the proof.
\end{proof}

\section{Special cases} \label{sec:special_cases} 



In this section we consider the particular cases that we need to complete the proof of our main result, starting with the first case of the list here above.

\begin{lemma}
Let $P$ be the width-$2$ interval order obtained by putting side by side $x$ intervals of length $1/x$ starting at $0$ and then $y$ intervals of length $1/y$ starting at $0$, where $x \geqslant y \geqslant 2$. If $x = y + 1$ then $\incompG{P}$ is a path with an odd number of vertices.
\end{lemma}

\begin{proof}
Note first that the number of vertices of $\incompG{P}$ is equal to $x + y$ $=$ $2y + 1$ and hence it is odd.

Since $\gcd(x,y) = 1$, we know that $\incompG{P}$ is connected. Moreover, the graph has at least one degree-$1$ vertex, namely the vertex whose interval starts at $0$ and is of length $1/x$. Thus, it suffices to show that the degree of each vertex is at most $2$. But this is clear because if a vertex of $\incompG{P}$ has degree $d \geqslant 3$, then the corresponding interval contains the intervals of at least $d-2$ of its neighbors. In particular, the interval is necessarily of length $1/y$ and we have $d \leqslant 3$. Furthermore, the endpoints of all intervals are located at integer multiples of $1/xy$. Thus if an interval of length $1/y = x \cdot (1/xy)$ contains one interval of length $1/x = y \cdot (1/xy) = (x-1) \cdot (1/xy)$ then it intersects exactly one other interval (and moreover both intervals either start at $0$ or end at $1$). This implies that $d \leqslant 2$. The result follows.
\end{proof}

\begin{lemma} \label{PaT} 
Let $P$ be a poset whose incomparability graph is a path with $n \geqslant 3$ vertices, with $n$ odd. Then 
$$
|P| \cdot \cent{P} \leqslant (2-\varepsilon) \log e(P)
$$ 
\end{lemma}

\begin{proof}
It is known that, if $\overline{G}(P)$ is an $n$-vertex path, then $e(P) = F_{n+1}$ with $n = | P |$ and $F_{n+1}$ the $(n+1)$-th Fibonacci number, see for example Atkinson and Chang \cite{Atk}. To compute $\cent{P}$ we use the KM algorithm, see Theorem~\ref{KM}. Assume without loss of generality that the bipartition $A, B$ of $\incompG{P}$ satisfies $|A| \geqslant |B|$.

Because $n = 2q + 1$ is odd, we find $k = 1$ and $|A_1| = |A| = q+1$, $|B_1| = |B| = q$ (we leave it to the reader the task of verifying this). Hence we have
$$
|P|  \cdot \cent{P} = (2q+1) \cdot h \left (\frac{q}{2q+1} \right ) = (q+1) \log \left ( \frac{2q+1}{q+1} \right) + q \log \left ( \frac{2q+1}{q}\right).$$

By a direct computation, we see that the inequality holds for $q \in \{1,2,3\}$. Notice in passing that the inequality is tight for $q = 1$. For $q=2$ the ratio is equal to $$\frac{3 \log(5/3) +  2\log(5/2)}{\log(8)} \simeq 1.62$$
From now on, we assume $q \geqslant 3$. From the easy lower bound $F_n \geqslant \phi^{n-2}$, where $n \geqslant 3$ and $\phi := \frac{1+ \sqrt{5}}{2}$ is the golden ratio, we obtain
$$
\frac{| P | \cdot H(\overline{G}(P))}{\log(e(P))} \leqslant \frac{(q+1) \log \left ( \frac{2q+1}{q+1} \right) + q \log \left ( \frac{2q+1}{q}\right)}{\log (\phi^{2q})} = \underbrace{\frac{\frac{q+1}{q}  \log \left ( \frac{2q+1}{q+1} \right) + \log \left ( \frac{2q+1}{q}\right)}{2 \log (\phi)}}_{=: f(q)}.
$$
Since $f'(x) < 0$ for every $x > 0$ we get $f(q) \leqslant f(3) \simeq 1.65 \leqslant 2 - \varepsilon$ for every $q \geqslant 3$.
\end{proof} 

This concludes the proof of case (1) in the proof of Theorem~\ref{thm:width-2}. The following lemma settles case (2).

\begin{figure}[h]
\label{partic:2}
\centering
\begin{tikzpicture}
\tikzstyle{vtx}=[circle,draw,fill=gray!25]

\draw (7,12) -- (8,13) -- (9,12);
\node[vtx] at (8,13) {};
\node[vtx] at (7,12) {};
\node[vtx] at (9,12) {};
\node at (8,12) {\ldots};
\end{tikzpicture}
\caption{$\incompG{P}$ in the particular case (2) $\ell = 1$, $\epochBsize_1 = 1$,  $\epochAsize_1 \geqslant 3$.}
\end{figure}

\begin{lemma}
\label{lem:case(2)}
Let $P$ be a poset whose incomparability graph is a star with $n \geqslant 3$ vertices. Then 
$$
|P| \cdot \cent{P} \leqslant (2-\varepsilon) \log e(P)
$$
\end{lemma}

\begin{proof}
Let $\epochAsize_1 = n - 1$ denote the number of leaves of the star and $\epochBsize_1 := 1$ (see Figure~\ref{partic:2}). We have $\log e(P) = \log (\epochAsize_1+1)$ and $$|P| \cdot \cent{P} = \epochAsize_1 \log \left (\frac{ \epochAsize_1+1}{\epochAsize_1} \right) + \log(\epochAsize_1+1).$$Now we are done since for $u \geqslant 2$: $$f(u) := \frac{u \log \left (\frac{u+1}{u} \right) + \log(u+1)}{\log(u+1)} \leqslant 2 - \varepsilon$$
Indeed, $f(2) = 2 - \varepsilon$ and for $u > 2$, the function $(1 + \frac{1}{u})^u$ is increasing and tends to the number $e$ for $u \to \infty$. Hence for $u \geqslant 3$ we have $(1 + \frac{1}{u})^u \leqslant e$ and so $$f(u) = 1 + \frac{u \log \left (\frac{u+1}{u} \right)}{\log(u+1)} \leqslant 1 + \frac{\log(e)}{\log(4)} \leqslant 1.73 \leqslant 2 - \varepsilon.$$
\end{proof}

Cases (3)--(6) in the proof of Theorem~\ref{thm:width-2} can be treated similarly as in Lemmas~\ref{PaT} and \ref{lem:case(2)}. We only summarize the main differences in Table~\ref{tble:partic} below. It is a straightforward task to turn the information in the table into a complete proof. We leave this to the reader.

\begin{table}[h]
\begin{center}
\begin{tabular}{|l|l|}
\hline
\multicolumn{2}{|c|}{\it Case 3: $\ell = 2$, $\epochBsize_i = 1$ and $\epochAsize_i \geqslant 2$}\\
\hline
\hline
$\displaystyle \log e(P) = \log (2\epochAsize_i+3)$ 
&$\displaystyle  |P| \cdot \cent{P} = 2 + \epochAsize_i \log \left (\frac{\epochAsize_i+1}{\epochAsize_i} \right) + \log(\epochAsize_i+1)$\\
\hline
$\displaystyle \frac{|P| \cdot  \cent{P}}{\log e(P)} \leqslant \frac{3\log(3)}{\log(7)} \leqslant 1.7$
&\mbox{}\raise-1ex\hbox{\begin{tikzpicture}[scale=.5, inner sep=2.5pt]
\tikzstyle{vtx}=[circle,draw,fill=gray!25]

\draw (5,13) -- (5,12) -- (7,13);
\draw (6,12) -- (7,13) -- (8,12);
\node[vtx] at (5,13) {};
\node[vtx] at (5,12) {};

\node[vtx] at (7,13) {};
\node[vtx] at (6,12) {};
\node[vtx] at (8,12) {};
\node at (7,12) {\ldots};
\node at (3,12.5) {$\incompG{P}$};
\end{tikzpicture}}\\

\hline
\multicolumn{2}{|c|}{We have $i=2$ in case $(|\epoch_{1}|, |\epoch_{2}|)$ is equal to $(2,m)$,}\\
\multicolumn{2}{|c|}{and $i=1$ in case $(|\epoch_{1}|, |\epoch_{2}|)$ is equal to $(m,2)$.}\\
\hline

\hline
\hline
\multicolumn{2}{|c|}{\it Case 4: $\ell = 3$, $\epochBsize_i = 1$ and $\epochAsize_i \geqslant 2$}\\
\hline
\hline
$\displaystyle \log e(P) = \log (4\epochAsize_i+8)$ 
&$\displaystyle  |P| \cdot \cent{P} = 4 + \epochAsize_i \log \left (\frac{\epochAsize_i+1}{\epochAsize_i} \right) + \log(\epochAsize_i+1)$\\
\hline
$\displaystyle \frac{|P| \cdot  \cent{P}}{\log e(P)} \leqslant \frac{2 + 3\log(3)}{4} \leqslant 1.7$
&\mbox{}\raise-1ex\hbox{\begin{tikzpicture}[scale=.5, inner sep=2.5pt]
\tikzstyle{vtx}=[circle,draw,fill=gray!25]

\draw (5,13) -- (5,12) -- (7,13);
\draw (6,12) -- (7,13) -- (8,12);
\draw (7,13) -- (9,12) -- (9,13);
\node[vtx] at (5,13) {};
\node[vtx] at (5,12) {};

\node[vtx] at (7,13) {};
\node[vtx] at (6,12) {};
\node[vtx] at (8,12) {};
\node[vtx] at (9,12) {};
\node[vtx] at (9,13) {};
\node at (7,12) {\ldots};
\node at (3,12.5) {$\incompG{P}$};
\end{tikzpicture}}\\

\hline
\multicolumn{2}{|c|}{Here we must have $i=2$.}\\
\hline

\hline
\hline
\multicolumn{2}{|c|}{\it Case 5: $\ell = 2$, $\epochBsize_i =  \epochAsize_i +1$}\\
\hline
\hline
$\displaystyle \log e(P) = \log (F_{2\epochBsize_i+4})$ 
&$\displaystyle  |P| \cdot \cent{P} = 2 + (\epochBsize_i+1) \log \left ( \frac{2\epochBsize_i+1}{\epochBsize_i+1} \right) + \epochBsize_i \log \left ( \frac{2\epochBsize_i+1}{\epochBsize_i}\right)$\\
\hline
$\displaystyle \frac{|P| \cdot  \cent{P}}{\log e(P)} \leqslant \frac{3\log(3)}{4 \log(\phi) } \leqslant 1.72$
&\mbox{}\raise-1ex\hbox{\begin{tikzpicture}[scale=.5, inner sep=2.5pt]
\tikzstyle{vtx}=[circle,draw,fill=gray!25]

\draw (5,13) -- (5,12) -- (7,13);
\draw (13,13) -- (14,12) -- (15,13) -- (16,12);
\draw (6,12) -- (7,13) -- (8,12) -- (9,13) -- (10,12) -- (11,13);
\node[vtx] at (5,13) {};
\node[vtx] at (5,12) {};

\node[vtx] at (7,13) {};
\node[vtx] at (6,12) {};
\node[vtx] at (8,12) {};
\node[vtx] at (9,13) {};
\node[vtx] at (10,12) {};
\node[vtx] at (11,13) {};
\node[vtx] at (13,13) {};
\node[vtx] at (14,12) {};
\node[vtx] at (15,13) {};
\node[vtx] at (16,12) {};
\node at (3,12.5) {$\incompG{P}$};
\node at (12,12.5) {$...$};
\end{tikzpicture}}\\
\hline
\multicolumn{2}{|c|}{We have $i=2$ in case $(|\epoch_{1}|, |\epoch_{2}|)$ is equal to $(2,m)$,}\\
\multicolumn{2}{|c|}{and $i=1$ in case $(|\epoch_{1}|, |\epoch_{2}|)$ is equal to $(m,2)$.}\\
\hline

\hline
\hline
\multicolumn{2}{|c|}{\it Case 6: $\ell = 3$, $\epochBsize_i =  \epochAsize_i +1$}\\
\hline
\hline
$\displaystyle \log e(P) = \log (F_{2\epochBsize_i+6})$ 
&$\displaystyle  |P| \cdot \cent{P} = 4 + (\epochBsize_i+1) \log \left ( \frac{2\epochBsize_i+1}{\epochBsize_i+1} \right) + \epochBsize_i \log \left ( \frac{2\epochBsize_i+1}{\epochBsize_i}\right)$\\
\hline
$\displaystyle \frac{|P| \cdot  \cent{P}}{\log e(P)} \leqslant \frac{2 + 3\log(3)}{6 \log(\phi) } \leqslant 1.7$
&\mbox{}\raise-1ex\hbox{\begin{tikzpicture}[scale=.5, inner sep=2.5pt]
\tikzstyle{vtx}=[circle,draw,fill=gray!25]

\draw (5,13) -- (5,12) -- (7,13);
\draw (15,13) -- (17,12) -- (17,13);
\draw (13,13) -- (14,12) -- (15,13) -- (16,12);
\draw (6,12) -- (7,13) -- (8,12) -- (9,13) -- (10,12) -- (11,13);
\node[vtx] at (5,13) {};
\node[vtx] at (5,12) {};

\node[vtx] at (7,13) {};
\node[vtx] at (6,12) {};
\node[vtx] at (8,12) {};
\node[vtx] at (9,13) {};
\node[vtx] at (10,12) {};
\node[vtx] at (11,13) {};
\node[vtx] at (13,13) {};
\node[vtx] at (14,12) {};
\node[vtx] at (15,13) {};
\node[vtx] at (16,12) {};
\node[vtx] at (17,12) {};
\node[vtx] at (17,13) {};
\node at (3,12.5) {$\incompG{P}$};
\node at (12,12.5) {$...$};
\end{tikzpicture}}\\

\hline
\multicolumn{2}{|c|}{Here we must have $i=2$.}\\
\hline
\end{tabular}
\end{center}

\caption{\label{tble:partic} For each one of cases (3)--(6) in the proof of Theorem~\ref{thm:width-2}, the table gives the expressions of both  $\log e(P)$ and $|P| \cdot \cent{P}$, an upper bound on the ratio and a drawing of the incomparability graph $\incompG{P}$.}
\end{table}

\section{Acknowledgments}

The two authors thank the anonymous referee for many useful remarks and suggestions that helped them to improve the presentation.

\end{document}